\documentclass{amsart}
\usepackage{amsmath, amsthm, amssymb}
\usepackage{verbatim}
\usepackage{tikz}
\usetikzlibrary{matrix}

\usepackage[mathscr]{eucal} 
\usepackage{mathrsfs} 
\usepackage{cancel}

\newcommand{\lc}{\left<}
\newcommand{\rc}{\right>}

\newcommand{\veps}{\varepsilon}

\newcommand{\al}{\alpha} 
\newcommand{\ze}{\zeta} 

\newcommand{\ga}{\gamma}
\newcommand{\de}{\delta}

\newcommand{\om}{\omega}

\newcommand{\Om}{\Omega}

\newcommand{\cP}{\mathcal{P}}

\newcommand{\cL}{\mathcal{L}}
\newcommand{\cM}{\mathcal{M}}

\newcommand{\cO}{\mathcal{O}}

\newcommand{\bB}{\mathbb{B}}
\newcommand{\bK}{\mathbb{K}}
\newcommand{\bP}{\mathbb{P}}
\newcommand{\bR}{\mathbb{R}}
\newcommand{\bS}{\mathbb{S}}

\newcommand{\bC}{\mathbb{C}}

\newcommand{\bN}{\mathbb{N}}
\newcommand{\vpi}{\varpi}

\newcommand{\ov}[1]{\overline{#1}}

\newcommand{\cali}[1]{\mathscr{#1}}

\newcommand{\Cc}{\cali{C}}

\newtheorem{thm}{Theorem}

\newtheorem{lem}[thm]{Lemma}
\newtheorem{cor}[thm]{Corollary}

\theoremstyle{definition}

\newtheorem{defn}[thm]{Definition}
\newtheorem{remark}[thm]{Remark}

\newtheorem{expl}[thm]{Example}

\numberwithin{thm}{section}
\numberwithin{equation}{section}

\renewcommand{\[}{\begin{equation}}
\renewcommand{\]}{\end{equation}}

\title[Fekete points on uniformly polynomially cuspidal sets]{Convergence speed for Fekete points on uniformly polynomially cuspidal sets}

\author{Hyunsoo Ahn}  \address{Department of Mathematical Sciences, KAIST, 291 Daehak-ro, Yuseong-gu, Daejeon 34141, South Korea} \email{kakapoolove@kaist.ac.kr}

\author{Ngoc Cuong Nguyen}  \address{Department of Mathematical Sciences, KAIST, 291 Daehak-ro, Yuseong-gu, Daejeon 34141, South Korea} \email{cuongnn@kaist.ac.kr}

\begin{document} 


\begin{abstract} We obtain the convergence speed for Fekete points on uniformly polynomially cuspidal compact sets introduced by  Paw\l ucki and Ple\'sniak. This is done by showing that these sets are $(\Cc^\al, \Cc^{\al'})$-regular in the sense of Dinh, Ma and Nguyen.
\end{abstract}


\maketitle

\section{Introduction}

Let $(X,\om)$ be a compact K\"ahler manifold of dimension $n$. A weighted compact subset $(K,\phi)$ consists of a non-pluripolar compact subset $K$ in $X$ and a real-valued continuous function $\phi$ on $K$.
 Let us denote $PSH(X,\om)$ the space of all $\om$-plurisubharmonic functions on $X$ ($\om$-psh for short). 
The weighted Siciak-Zaharjuta extremal function associated to $(K,\phi)$ is the upper semi-continuous regularization $V_{K,\phi}^*$ of 
$$
	V_{K,\phi}(x) = \sup\{v(x): v \in PSH(X,\om), v\leq \phi \text{ on }K\}.
$$
The normalized Monge-Amp\`ere measure of $V_{K,\phi}^*$ is called the equilibrium measure of $(K,\phi)$, namely,
$$
	\mu_{\rm eq} (K,\phi) := \frac{(\om + dd^c V_{K,\phi}^*)^n}{\int_X\om^n},
$$
where the right hand side used the Bedford-Taylor wedge product for bounded $\om$-psh functions \cite{BT82}. 
The weighted extremal function is inspired by Siciak \cite{Si62, Siciak82} and Zaharjuta \cite{Za76} which plays a crucial role in pluripotential theory on compact K\"ahler manifolds. There is a large literature focusing on the study of this function and its applications. We refer the readers to the books by Klimek \cite{Kl91} for the local setting in $\bC^n$ and by Guedj and Zeriahi \cite{GZ17}  for the global one (see also Dinh-Sibony 
\cite{DS06}), where they contain a comprehensive list of its properties and applications.

Assume now $X$ is projective, $L$ is an ample holomorphic line bundle on  $X$ and $[\om] \in c_1(L)$.  Let $N_q = \dim H^0(X, L^q)$ with $q\geq 0$ integer, a basis $\{s_1,...,s_{N_q}\}$ of $H^0(X,L^q)$ and  $$P = (p_1,...,p_{N_q}) \in K^{N_q}$$
be a configuration of points in the given compact subset $K$. Then, following  \cite{BBW11}  the ordered set of points $P$ is called a Fekete configuration of order $q$ for $(K,\phi)$ if it maximizes the  Vandermonde-type determinant
$$
	|\det (s_i(x_j))| e^{- (\phi(x_1)+\cdots + \phi(x_{N_q}))}
$$
over $(x_1,...,x_{N_q}) \in K^{N_q}$. This condition is independent of the choice of the basis $(s_i)_{i=1}^{N_q}.$ Moreover, for such a Fekete configuration $P$ we call
$$
	\mu_q := \frac{1}{N_q} \sum_{j=1}^{N_q} \de_{p_{j}} 
$$
to be the Fekete measure of order $q$. 

A fundamental result obtained by Berman, Boucksom and Witt-Nystrom \cite{BBW11} is that the convergence
\[\label{eq:weak-c}
	\lim_{q\to \infty} \mu_{q} = \mu_{\rm eq} (K,\phi)
\]
holds in weak topology of measures. 
This is a generalization of the classical result for $n=1$ (see e.g., \cite[Theorem~1.3]{ST97}, \cite{BBL92} and \cite{BSV89}). In particular, it answers affirmatively an open question asked by Siciak \cite[Problem~15.3]{Siciak82} in the local setting (see also \cite[Problem~3.3]{ST97}). There is available a self-contained proof of this local result using only weighted pluripotential theory by  Levenberg \cite{Le10}, which is derived from \cite{BB10} and \cite{BBW11}. Most recently, the expository of Dujardin \cite{Du20} contains motivations, the applications of the result as well as many references. 

We are now interested in the speed of convergence of such sequences in \eqref{eq:weak-c}.
Let $\cM(X)$ be the space of probability measure on $X$. For $\ga>0$ the distance ${\rm dist}_\ga$ between two measures $\mu, \nu$ in $\cM(X)$ is given by
\[\label{eq:dist}	
	{\rm dist}_\ga (\mu, \nu) = \sup_{\|v\|_{\ga} \leq 1} | \lc \mu -\nu, v\rc|,
\]
where $v$ is smooth real-valued function and $\|\cdot \|_{\ga}$ is the $\ga$-H\"older norm on $X$. This is a generalization of ${\rm dist}_1$ for $\ga=1$ the classical Kantorovich-Wasserstein distance. 

The sharp speed of convergence of the Fekete measures is showed firstly by Lev, Ortega-Cerda \cite{LO16} for $\phi$ smoothly strictly $\om$-plurisubharmonic on  $X$. Later,  Dinh, Ma and Nguyen \cite{DMN} obtained the speed of convergence for a very large family of compact sets on $X$ and general H\"older continuous quasi-plurisubharmonic weights $\phi$ in term of the distance \eqref{eq:dist}. The family contains compact domains whose boundary are $C^2$-smooth in $X$. The expectation in \cite[page 562]{DMN} is that the estimate would hold for for all uniformly polynomially cuspidal (UPC) sets introduced by Paw\l ucki and Ple\'sniak  \cite{PP86} (see Section~\ref{sec:upc} and Definition~\ref{defn:UPC-mfd} below).  Our main result is to confirm this expectation.

\begin{thm}\label{thm:main} Let $K\subset X$ be a compact UPC set and $0<\ga\leq 2.$ Let $\phi$ be a $\al$-H\"older continuous function on $K$, where $0<\al \leq 1$. Then, there exist uniform constants $c>0$ and $0<\al'<1$ depending on $K, \al$ and H\"older norm of $\phi$ such that 
$$
	{\rm dist}_{\ga} (\mu_d, \mu_{\rm eq}(K,\phi)) \leq \frac{  c [\log d]^{3\al''}}{d^{\al''}} \quad \forall d> 1, 
$$ 
where $\al'' = \ga \al' / (24 +12\al')$.
\end{thm}

Notice that we can be computed explicitly the constant $\al'$ in terms of $\al$ and the UPC property of $K$ (see Remark~\ref{rem:al'} below).
Originally, the UPC sets in \cite{PP86} related with  studying the H\"older property of Siciak-Zaharjuta extremal function in pluripotential theory. It was pointed out there that this family is very large and nearly optimal for the H\"older property. Every bounded convex domain in $\bR^n$ (or $\bC^n\equiv \bR^{2n}$) or every bounded Lipchitz domain are UPC. This family  is essentially larger than the family of  fat subanalytic sets (these sets are fundamental objects in real algebraic geometry \cite{BM88}). Furthermore, the boundary of UPC sets may exhibit very irregular behavior (see Examples~\ref{expl:not-subanalytic}, \ref{expl:comb} below).  The study of UPC sets has been being very active since the appearance of \cite{PP86}. We refer the reader to a survey by Pl\'esniak \cite{Pl06} for the state-of-the-art of  the results and applications of such sets in approximation theory and the extension problem.

Notice that the estimate in the theorem has been proved for fat subanalytic sets in $\bR^n$ in \cite[Theorem~5.2]{N24}.  On the other hand, Vu \cite{Vu18} proved the theorem for $K$ being (real) smooth generic submanifolds, for example $\bS^n$.

An immediate consequence of Theorem~\ref{thm:main} in 
the classical setting of a non-pluripolar compact set $K\subset\bC^n$,  considering it as an affine coordinate chart of $\bP^n$, is as follows. It was pointed out in \cite{BBW11} (see also \cite[Section~6]{BB10} and \cite{Du20}) that this is a special case of $X = \bP^n$ with $L = \cO(1)$ the tautological line bundle and $\om =\om_{FS}$ the Fubini-Study metric. 
Now the space $PSH(\bP^n,\om_{FS})$ is corresponding to the Lelong class
\[\label{eq:L-class}\notag
	\cL = \left\{ f \in PSH(\bC^n):  f (z) - \rho(z) < c_f\right\},
\]
where $\rho = \frac{1}{2}\log (1+ |z|^2)$ the potential of  $
\om_{FS}$ in $\bC^n$. We consider the (classical) Siciak-Zaharjuta extremal function
$$
	L_K(z) = \sup\{v(z): v\in \cL, \; v\leq 0 \text{ on }K\}.
$$
Then, the equilibrium measure associated to $K$ is given by
\[	
	\mu_{\rm eq}(K) =\frac{(dd^c L_K^*)^n}{\int_{\bC^n} (dd^c L_K^*)^n}. 
\]
Let $\cP_d(\bC^n)$ be the set of complex valued polynomials of degree at most $d$. Then its dimension is $N_d = \binom{n+d}{n}$. Let $\{e_1,...,e_{N_d}\}$ be an ordered system of all monomials $z^{\al}:= z^{\al_1}_1 \cdots z_n^{\al_n}$ with $|\al| = \al_1 +\cdots \al_n \leq d$, where $\al_i \in \bN$. For each system $x^{(d)} = \{x_1, ..., x_{N_d}\}$ of $N_d$ points of $\bC^n$ we define the generalized Vandermonde-type determinant ${\rm VDM} (x^{(n)})$ by
$$
	{\rm VDM} (x^{(d)}):= \det [e_i(x_j)]_{i,j=1,...,N_d}.
$$
Then, a {\em Fekete configuration of order $d$} for $K$ is  a system $\xi^{(d)}=\{\xi_{1},....,\xi_{N_d}\}$ of $N_d$ points of $K$ that maximizes the function $|{\rm VDM}(x^{(d)})|$ on $K$, i.e.,
$$	
	\left| {\rm VDM}(\xi^{(d)}) \right| = \max \left\{ \left|{\rm VDM}(x^{(d)}) \right|: x^{(d)}\subset K\right\}.
$$
Given a Fekete configuration 
$\xi^{(d)}$ of $K$, we  consider the probability measure on $\bC^n$ defined by
$$
	\mu_d := \frac{1}{N_d} \sum_{j=1}^{N_d} \de_{\xi_j},
$$
where $\de_x$ denotes the Dirac measure concentrated at the point $x$. Then,
the weak convergence \eqref{eq:weak-c} reads
$$
	\lim_{d\to \infty} \mu_d = \mu_{\rm eq}(K).
$$
In this setting we have

\begin{cor}\label{cor:local} Let $K$ be a uniformly polynomially cuspidal set in either $\bR^n$ or $\bC^n$. Let $\ga, \al'$ be the constants in Theorem~\ref{thm:main} applied for $\phi\equiv 0$ and $\al=1$. Denote $\al''= \frac{\ga \al'}{24+12\al'}$.  Then, there exists a uniform constant $C = C (K,\ga)$ such that 
$$
	|\lc \mu_d - \mu_{\rm eq}(K), v \rc| \leq  \frac{ C \|v\|_{\ga} [\log d]^{3\al''}}  {d^{\al''}} \quad \forall d > 1,
$$
for every Fekete measure $\mu_d$ of order $d$ and every test function $v\in C^{\ga}(\bC^n)$.
\end{cor}

Notice that in both Theorem~\ref{thm:main} and Corollary~\ref{cor:local} the speeds of convergence are far from being optimal.

\bigskip

{\em Acknowledgement.}  The last part of the paper has been completed while the second author was visiting at VIASM (Vietnam Institute for Advanced Study in Mathematics). He would like thank to the institute for the great hospitality and excellent working conditions. 

\section{Uniformly polynomially cuspidal sets and Fekete points}

\subsection{Uniformly polynomially cuspidal sets}
\label{sec:upc}  We are considering the compact sets with cusps in $\bK^n$, where here and throughout the note $$\bK \text{ is either }\bR \text{ or }\bC.$$ We also consider $\bR^n = \bR^n + i\cdot 0$ as a natural subset in  $\bC^n = \bR^{2n}$. Following \cite{PP86}  a compact subset $E \subset \bK^n$  is called {\em uniformly polynomially cuspidal} (UPC for short) if there exist positive constants $M$, $m$ and a positive integer $d$ such that for each $x\in E$, one may choose a polynomial map 
$$
	h_x: \bK \to \bK^n, \quad \deg h_x \leq d
$$
satisfying
$$\begin{aligned}
&	h_x(0) = x \quad \text{and}\quad h_x([0,1]) \subset E; \\
&	{\rm dist } (h_x(t), \bK^n \setminus E) \geq M t^m  \quad \text{for all } x\in E,\text{ and } t\in [0,1].
\end{aligned}$$

An important property of the UPC sets is that if $a \in E$, then 
\[ \label{eq:cusp}
	 E_a := \bigcup_{0\leq t \leq 1}  D(h_a(t), Mt^m) \subset E,
\]
where 
$D(p, r) = \{ x\in \bK^n: |x_1-p_1| \leq r, ..., |x_n-p_n| \leq r\}$ denotes the closed cube. Notice that $E_a$ is also a closed subset.  The UPC sets are fat, i.e., ${\rm int }E$ is dense in $E$. A very general family of  UPC sets is the one of fat subanalytic sets which was shown in \cite{PP86}.

It was also observed  by Paw\l ucki and Ple\'sniak   \cite[Proposition~1.2]{PP88} that the compact UPC sets are preserved under diffeomorphisms  in a neighborhood of $E$. In particular, the UPC sets in $\bC^n$ (i.e. $\bK =\bC$)  are well-defined on  complex manifolds of dimension $n$. We have a more general definition that include the case $\bK =\bR$.

\begin{defn}\label{defn:UPC-mfd} Let $E\subset X$ be a compact set. It is said to be uniformly polynomially cuspidal if there exists finitely many coordinate balls $\tau_i : \Om_i \to \bB(0,1)\subset \bC^n$ such that $E \subset \cup_i \Om_i$ and the sets $\tau(E \cap \Om_i) \cap \ov\bB(0,1/2)$ are  UPC in  $\bK^n$.
\end{defn}

Thus, this definition contains all UPC sets in $\bR^n$ considered as subset of $\bC^n\subset \bP^n$. Let us recall an example of UPC sets from \cite[Example~7.1]{PP86} which are not subanalytic. 

\begin{expl}\label{expl:not-subanalytic} Define $U = \{(x,y) \in \bR^2: 0<x\leq 1,\; 0<y<e^{-1/x}\}$ and take $E = [0,1]\times [-1,1] \setminus U$. Clearly, it is not subanalytic but  $E$ is UPC which satisfies even the property (P) introduced in \cite{Siciak85} (see also \cite[Definition~4.4]{N24}).  
\end{expl}

Another example \cite[Example~7.2]{PP86}, which is not subanalytic either, shows the very irregular behavior such as comb. 

\begin{expl}\label{expl:comb} Let $\{a_k\}$ and $\{\veps_k\}$ be strictly decreasing sequences of positive numbers both tending to zero such that 
$a_k - a_{k+1} > \veps_k + \veps_{k+1}$. Let $$E_k= \{(x,y)\in \bR^2: 0\leq x<y,|y-a_k|<\varepsilon_k\}$$ and define
$$
	E = [0,1] \times [-1,1] \setminus \bigcup_{k=2}^\infty E_k.
$$
Then, the comb $\bar E$ is not subanalytic. It does even not satisfy the property (P) as in Example~\ref{expl:not-subanalytic} for  $a_k =1/k$ and $\veps_k = e^{-1/k^2}/2$. However, $\bar E$ is UPC.
\end{expl}

\subsection{Convergence speed of Fekete measures for regular sets}
We will show that the UPC sets possessing a stronger local H\"older property.
Let us first recall the new regular property for compact sets introduced in \cite{DMN}. 

\begin{defn}\label{defn:dmn-reg}
A compact set $K$ is called {\em $(\Cc^\al, \Cc^{\al'})$-regular} with $0<\al,\al'\leq 1$ if for every $\al$-H\"older continuous function $\phi$, then $V_{K,\phi}$ is $\al'$-H\"older continuous whose $\al'$-H\"older norm depends only on $X, \al, \al'$ and $\|\phi\|_{\al}$.
\end{defn}

If $K = X$ the whole manifold then $V_{X,\phi} = P_X\phi$ is the envelope. As it is shown in \cite{B19} and \cite{To18} that the optimal regularity of $P_X\phi$ regularity is $C^{1,1}$  for a $C^2$-smooth weight. Therefore, in this case the range of $\al, \al'$ can be bigger than $1$. 

Our main focus is on a proper compact set  $K\subset \subset X$ where the best  regularity of $V_K$ is Lipschitz one \cite{Kl91}. So it is convenient to consider only the range $0< \al, \al' \leq 1$. If $K$ is $(\Cc^\al, \Cc^{\al'})$-regular, then clearly $V_K$ is $\al'$-H\"older continuous. Furthermore, if $K$ is $(\Cc^{\al}, \Cc^{\al'})$ regular, then it is $(\Cc^{\tau}, \Cc^{\al'})$ regular for every $\al \leq \tau \leq 1$ as $X$ is a compact manifold.

 Let us fix $0<\al, \al', \ga \leq 1$ and  a  $(\Cc^\al,\Cc^{\al'})$-regular compact set $K$. Let $\phi$ be a $\al$-H\"older continuous function on $X$. Then, \cite[Theorem~1.5]{DMN}  showed that  there is a uniform constant $C = C(K,\ga, \|\phi\|_\al)>0$ such that
\[\label{eq:speed-dmn}
	{\rm dist}_{\ga} (\mu_d, \mu_{\rm eq}(K,\phi)) \leq C \;  \frac{[\log d]^{3\al'}}{d^{\al' }} \quad \forall d>1,
\]
where $\al' = \ga \al / (24 +12\al)$.

Previously, the $(\Cc^\al,\Cc^\al)$-regularity with $0<\al<1$ was proved for every $C^2$-smooth bounded domain in $X$ in \cite{DMN} (see also \cite[Theorem~3.11]{MV22} for an improvement). Later, Vu \cite{Vu18} proved this property for smooth generic real manifolds in $X$. This family contains all totally real submanifolds as important examples.  This regularity is also useful for studying the large deviation principle for some beta ensembles  \cite{DN18} and estimating the Bergman functions \cite{MV22}.

Very recently, the H\"older regularity of the extremal function has been studied in \cite{N24}. It contains an effective criterion to check the regularity in Definition~\ref{defn:dmn-reg}. For example, it follows from by \cite[Corollary~4.10, Lemma~5.1]{N24} that all compact fat subanalytic sets in $\bR^n$ are $(\Cc^\al,\Cc^{\al'})$-regular. We will briefly recall the criterion here.  Recall that the Lelong class in $\bC^n$ is given by
\[\label{eq:L-class}\notag
	\cL = \left\{ f \in PSH(\bC^n):  f (z) - \rho(z) < c\right\},
\]
where $\rho = \frac{1}{2}\log (1+ |z|^2)$ and the constant $c$ may depend on $f$. For a  non-pluripolar compact subset $E$ in $\bC^n$,
$$
	L_E(z) := \sup \left\{f(z): f \in \cL,\;  f  \leq 0 \quad \text{on } E\right\}.
$$
It is a well-known fact that $E$ is non-pluripolar if and only if $L_E^*$ is locally bounded in $\bC^n$. Let us state a useful inequality of $L_E$ under holomorphic polynomials \cite[Lemma~4.11]{N24} that will be used in the sequel.
  
\begin{lem}\label{lem:L-polynomial} Let $E \subset \bC^k$ be a compact subset and $h: \bC^k \to \bC^n$ be a complex valued polynomial mapping of degree $d$. Then, for every $w\in \bC^k$,
$$
	L_{h(E)}(h(w)) \leq d \cdot L_{E} (w).
$$
\end{lem}

Our main interest is to study higher regularity of the extremal function. 
The modulus of continuity of $L_E$ at $a\in E$ is given by  
$$
	\vpi_E'(a, \de) = \sup_{|z-a| \leq \de} L_E(z)
$$
for $0< \de \leq 1$. Then, $L_E$ is continuous at $a$ if and only if 
$\lim_{\de\to 0} \vpi'_E(a,\de) =0$. 
Put
$$
	\vpi_E'(\de) = \sup\{ \vpi_E'(a,\de): a\in E\}
$$
which is the modulus of continuity of $L_E$ over $E$. 
It is a well-known fact due to B\l ocki \cite[Proposition~3.5]{Siciak97} that the modulus of continuity of $L_E$ on $E$ controls the modulus of continuity of $L_E$ on $\bC^n$. Namely,
\[\label{eq:local-modulus-of-continuity}
	|L_E(z) - L_E(w)| \leq \vpi'_E(|z-w|), \quad z,w\in \bC^n, \; |z-w|\leq 1.
\]
Inspired from the definition of local $L$-regularity in \cite{Siciak81}  we introduced in \cite[Definition~3.2]{N24} the following notion  

\begin{defn}\label{defn:local-mu-hcp} Let $q\geq 0$ be an integer and  $0< \mu \leq 1$. Let $B(a,r)\subset \bC^n$ be a closed ball with center at $a$ and of radius $r>0$. We say that  a compact subset $K\subset \bC^n$  has
\begin{itemize}
\item
[(a)]  local $\mu$-H\"older continuity property (local {\rm $\mu$-HCP} for short) of order $q$ at $a\in K$ if there exist constants  $C>0$ and $0< r_0 \leq 1$  such that 
$$
	\vpi_{K\cap B(a,r)}' (a,\de) \leq \frac{C  \de^\mu}{r^q}, \quad 0<\de\leq 1,\, 0< r <r_0;
$$
\item
[(b)] local $\mu$-HCP of order $q$ if it has local $\mu$-HCP of order $q$ at every point $a\in K$ for the constants $C, r_0$ being  independent of $a$.
\end{itemize}
\end{defn}
Here it is important to know both the H\"older exponent and the H\"older coefficient. The  basic examples of local $\mu$-HCP sets are due to Siciak \cite{Siciak85} (see also \cite[Corollary~4.6]{N24}).
\begin{lem}[Siciak]
\label{lem:convex}
A convex compact subset in $\bK^n$, where $\bK = \bR$ or $\bK=\bC$, with non-void interior  has  local {\rm $\mu$-HCP} with the (optimal) exponent $\mu=1/2$ and of order $q= n$. 
\end{lem}

Notice that the proof of the characterization \cite[Theorem~1.2]{N24} showed that H\"older norm of $V_{K,\phi}$ depends only on $K, \om$ and the H\"older norm of $\phi$.
Thus, we have the sufficient condition for the $(\Cc^\al,\Cc^{\al'})$-regularity in \cite[Lemma~5.1]{N24}.

\begin{cor}\label{cor:regularity} Let $K \subset X$ be a non-pluripolar compact subset.  If $K$ has local H\"older continuous property of order $q$, then $K$ is ($\Cc^{\al}, \Cc^{\al'}$)-regular for some $0<\al, \al' \leq 1$. \end{cor}

It would be interesting to know whether  $(\Cc^\al, \Cc^{\al'})$-regularity implies the local HCP property. This is the case for compact sets satisfying additional geometric condition \cite[Theorem~1.2-(b)]{N24}.

\section{Proof of Theorem~\ref{thm:main}}

We need the following observation due to Pierzcha\l a \cite[Lemma~3.1]{Pi05} which helps to remove the extra assumption in \cite[Theorem~4.9]{N24} (see also  \cite[Remark~4.8-(b)]{N24}).

\begin{lem} \label{lem:coefficients} Let $E$ be a compact UPC set in $\bK^n$ and let $h_x(t) = \sum_{k=0}^d a_k(x) t^k$ be the polynomial satisfying the UPC condition, where $(x, t) \in E\times [0,1]$ and  $a_k: E \to \bK^n$. Then, $a_k(E)$ is a bounded set for each $k=0,...,d$.
\end{lem}

\begin{proof}  Let $0 \leq k\leq d$ and  $x \in E$. We write $a_k(x) = (a_k^1(x),....,a_k^n(x))$ and $h_x(t) = (h_x^1(t),...,h_x^n(t))$. It is enough to show that each $a_k^s(x)$, $s=1,...,n$ is uniformly bounded. Fix such an $s$, we have for $t_j = 1/j \in [0,1]$ and $j=1,....,d+1$,  the system of $(d+1)$ equations for a unknown vector $(a_0^s(x),...,a_d^s(x)) \in \bK^{d+1}$
$$
	\sum_{k=0}^d  \frac{a_k^s(x)}{j^k} = h_x^s(1/j) \in \pi_s(E), \quad j=1,....,d+1,
$$
where $\pi_s: \bK^n\to \bK$ the projection to the $s$-th coordinate.
Since $\pi_s(E)$ is bounded,  the unique solution $(v_0,...,v_d)$ of the equations is uniformly bounded.
\end{proof}

Thanks to Corollary~\ref{cor:regularity} and the estimate \eqref{eq:speed-dmn}, to prove Theorem~\ref{thm:main} it is enough to verify that a compact UPC set has local $\mu$-HCP of some order. To this end we follow the proof in \cite[Theorem~4.9]{N24}, which is based on the one of \cite{PP86}. For the sake of completeness we give detail proof here for the sets being either in $\bR^n$ or $\bC^{n} \equiv \bR^{2n}$.

\begin{thm}\label{thm:UPC} Let $E \subset X$ be a compact {\rm UPC} subset. Then, $E$  has local {\rm $\mu$-HCP} of some order $q$.
\end{thm}

\begin{proof} Notice that the local $\mu$-HCP are invariant under biholomorphic maps. Therefore, by the definition of UPC sets (Definition~\ref{defn:UPC-mfd}) we can assume without loss of generality that $E$ is a compact UPC set in $\bK^n$, where $\bK$ is either $\bR$ or $\bC$. In what follows if $\bK = \bR$, then
the space $\bR^n$ is identified with the subset $\bR^n+i\cdot 0$ of  $\bC^n$. 

Let $a\in E$ be fixed and denote by $D(a,r)$ a closed polydisc. 
Observe first that the set $E_a$ defined in \eqref{eq:cusp} satisfies
$$	E_a = \{x\in \bK^n: x = h(t) + M t^m \left(u_1^m,...,u_n^m \right), t\in [0,1], |x_i| \leq 1, i=1,...,n\}.
$$
Let $S\subset \bR \times \bR^{n}$ be the pyramid 
$$
	S = \{(t, tu_1,...,tu_n) \in \bR \times \bR^{n}:  |t| \in [0, 1], |u_i| \leq 1, i=1,...,n\}.
$$
This is a convex set (with non-void interior in $\bR^{n+1}$) which implies that it has HCP by Lemma~\ref{lem:convex}. The crucial observation is that
our cusp $E_a$ contains the image of this set under the polynomial projection
$p(t,z): \bC \times \bC^{n} \to \bC^n$ given by
$$
	p(t,z) = h(t) + M (z_1^m, ..., z_n^m).
$$
Clearly, $p(S) \subset E_a$ and $p(0,0) = h(0) =a$.

To show the {\em local} $\mu$-HCP of some order at $a$  we need to shrink a bit that pyramid. We claim that for each $0< r \leq 1$, we can find $0< r' \leq r$ such that a smaller pyramid
$$
	S(r') := \{ (t,tu) \in \bR \times \bR^{n}: |t|\in [0,r'], |u_1| \leq 1, ..., |u_n| \leq 1\}
$$
satisfies
\[\label{eq:s-pyramid}
	p(S(r'))  \subset E_a \cap D(a,r).
\]
Indeed, for $(t,tv) \in S(r') \subset S$, the point $x = p(t, tv)= h(t) + M t^m\cdot v^m \in E_a$. Moreover,
$$\begin{aligned}
|x-a| &= |h(t) + M t^m\cdot v^m - h(0)| \\
& \leq  |h(t) - h(0)| + M t^m |v|^m \\
&\leq \left(\sum_{\ell=0}^d |h^{(\ell)}(0)| \right) r' +  M r', 
\end{aligned}
$$
where we used the fact that $m$ is a positive integer. Since  $h_x(t) = \sum_{\ell=0}^d b_\ell(x) t^\ell$,  we can choose
\[ \label{eq:order-UPC}
	r' = \frac{r}{1+ d!\sum_{\ell=0}^d \|h^{(\ell)}(0)\| + M} \quad\text{where }
	\|h^{(\ell)}(0)\| := \sup_{x\in E} |b_\ell(x)|.
\]
Thanks to Lemma~\ref{lem:coefficients} the uniform bound $d!\sum_{\ell=0}^d \|h^{(\ell)}(0)\|$ for the sum $\sum_{\ell=0}^d |h^{(\ell)}(0)|$ does not depend on the point $a$. 

Since $S(r')$ contains a ball of radius $ \tau_n r'$ in $\bR^{n+1}$ with a numerical constant $\tau_n$, it follows from Lemma~\ref{lem:convex} that $S(r')$ has local $\frac{1}{2}$-HCP of order $q=n+1$, i.e.,
\[\label{eq:HCP-convex}
	L_{S(r')} (t,v) \leq \frac{C \de^\frac{1}{2}}{r'^{n+1}}
\]
for every $(t,v) \in S_\de (r') := \{ \ze \in \bC^{n+1}: {\rm dist} (\ze, S(r')) \leq \de\}$ and $C$ does not depend on $r'$ and $\de$.

Moreover, for all $0< \de \leq  r'$, we have   $$P(\de) = \{(t,z) \in \bC\times \bC^n: |t|\leq \de, |z_i| \leq \de, i=1,...,n\} \subset S_\de(r').$$ 
Then, for such a small $\de$, the following inclusions hold
\[\label{eq:inclusion-chain}
	B (a, M \de^m) \subset p (\{0\}\times D(0,\de))\subset p(P(\de)) \subset p ( S_\de (r')).
\]

Now we are ready to conclude the local $\mu$-HCP of $ E_a$. Let $z\in \bC^n$ be such that 
\(z\in B(a,M\delta^m)\).
By \eqref{eq:inclusion-chain} we have $z = p(t,v) \in \bC^n$ for some $(t,v)\in S_\de (r')$. Furthermore, by \eqref{eq:s-pyramid} we have $F:= p(S(r')) \subset E_a \cap D(a,r)$. Combining these facts  with \eqref{eq:HCP-convex} we obtain
$$\begin{aligned}
L_{ E_a\cap D(a,r)}(z) &\leq L_{F} (z) \\
&= L_{F} (p(t,v)) \\
&\leq \max(d,m) \cdot L_{S(r')} (t,v)  \\
& \leq \frac{C \de^\frac{1}{2}}{r'^{n+1}},
\end{aligned}$$
where for the third inequality we used Lemma~\ref{lem:L-polynomial} and the last constant $C$ does not depend on $r'$ and $a$.
Rescaling $\de := M\de^m \leq r'$, we obtain
$$
	L_{E_a \cap D(a,r)} (z) \leq \frac{C \de^\frac{1}{2m}}{r'^{n+1}}
$$
for every ${\rm dist}(z,a) \leq \de$, where $0<\de \leq r'$. Notice that $r$ and $r'$ are comparable by \eqref{eq:order-UPC}. Hence, $E_a$ has local $\mu$-HCP at $a$ with the exponent $\mu = 1/2m$ and of order $q=n+1$ and so does $E\supset E_a$. This finishes the proof of the theorem.
\end{proof}

\begin{remark}\label{rem:al'} Let $\mu$ and $q$ be the constants in the proof of Theorem~\ref{thm:UPC} for the compact UPC set $K$. Then, it follows from \cite[Theorem~1.2]{N24} that we can choose  $\al'$ in Theorem~\ref{thm:main} as
$$\al' =  \frac{\tau^2}{\tau+2+q}\quad \text{where } \tau = \min\left\{\al,\frac{\mu}{1+q}\right\}.$$  

\end{remark}

\end{document}